\documentclass[reqno,10pt,centertags]{amsart} 
\usepackage{amsmath,amsthm,amscd,amssymb,latexsym,esint,upref,stmaryrd,
enumerate,color,verbatim,yfonts}
\usepackage{color}
\usepackage{hyperref} 
\newcommand*{\mailto}[1]{\href{mailto:#1}{\nolinkurl{#1}}}
\newcommand{\arxiv}[1]{\href{http://arxiv.org/abs/#1}{arXiv:#1}}

\newcommand{\bbC}{{\mathbb{C}}}

\newcommand{\bbN}{{\mathbb{N}}}

\newcommand{\bbR}{{\mathbb{R}}}

\newcommand{\bsA}{{\boldsymbol{A}}}

\newcommand{\bsD}{{\boldsymbol{D}}}

\newcommand{\bsH}{{\boldsymbol{H}}}
\newcommand{\bsI}{{\boldsymbol{I}}}

\newcommand{\cA}{{\mathcal A}}
\newcommand{\cB}{{\mathcal B}}
\newcommand{\cC}{{\mathcal C}}

\newcommand{\cH}{{\mathcal H}}

\newcommand{\beq}{\begin{equation}}
\newcommand{\enq}{\end{equation}}

\DeclareMathOperator{\supp}{supp}

\DeclareMathOperator{\dom}{dom}

\DeclareMathOperator{\tr}{tr}

\DeclareMathOperator*{\slim}{s-lim}

\renewcommand{\Im}{\text{\rm Im}}

\newcommand{\loc}{\operatorname{loc}}

\newcommand{\no}{\notag}
\newcommand{\lb}{\label}
\newcommand{\f}{\frac}

\newcommand{\ol}{\overline}

\newcommand{\wti}{\widetilde}
\newcommand{\Oh}{O}

\newcommand{\bi}{\bibitem}

\let\geq\geqslant
\let\leq\leqslant

\makeatletter
\def\theequation{\@arabic\c@equation}

\allowdisplaybreaks 
\numberwithin{equation}{section}

\newtheorem{theorem}{Theorem}[section]

\newtheorem{lemma}[theorem]{Lemma}

\newtheorem{example}[theorem]{Example}

\theoremstyle{remark}
\newtheorem{remark}[theorem]{Remark}

\begin{document}

\title[Almost Analytic Extensions To Operator Bounds]{Some Applications Of Almost Analytic Extensions To Operator Bounds in Trace Ideals} 

\author[F.\ Gesztesy]{Fritz Gesztesy} 
\address{Department of Mathematics, 
University of Missouri, Columbia, MO 65211, USA}
\email{\mailto{gesztesyf@missouri.edu}}
\urladdr{\url{http://www.math.missouri.edu/personnel/faculty/gesztesyf.html}}

\author[R.\ Nichols]{Roger Nichols}
\address{Mathematics Department, The University of Tennessee at Chattanooga, 
415 EMCS Building, Dept. 6956, 615 McCallie Ave, Chattanooga, TN 37403, USA}
\email{\mailto{Roger-Nichols@utc.edu}}
\urladdr{\url{http://www.utc.edu/faculty/roger-nichols/index.php}}

\dedicatory{Dedicated with deep admiration to Yurij Makarovich Berezansky \\ on the occasion of 
his 90th birthday.}

\date{\today}
\thanks{R.N. gratefully acknowledges support from an AMS--Simons Travel Grant.} 
\thanks{Appeared in {\it Methods Funct. Anal. Topology}, {\bf 21}, 151--169 (2015).}
\subjclass[2010]{Primary: 47A30, 47A53, 47A60, 47B10; Secondary: 47B25, 30D99, 35Q40.}
\keywords{Almost analytic extensions, trace ideals, operator bounds.}

\begin{abstract} 
Using the Davies--Helffer--Sj\"ostrand functional calculus based on almost analytic extensions, 
we address the following problem: Given self-adjoint operators $S_j$, $j=1,2$, in 
$\cH$, and functions $f$ in an appropriate class, for instance, $f \in C_0^{\infty}(\bbR)$, 
how to control the norm $\|f(S_2) - f(S_1)\|_{\cB(\cH)}$ in terms of the norm of the difference of  resolvents, $\big\|(S_2 - z_0 I_{\cH})^{-1} - (S_1 - z_0 I_{\cH})^{-1}\big\|_{\cB(\cH)}$, for 
some $z_0 \in \bbC\backslash\bbR$. We are particularly interested in the case where $\cB(\cH)$ is replaced by a trace ideal, $\cB_p(\cH)$, $p \in [1,\infty)$.
\end{abstract}

\maketitle



\section{Introduction} \lb{s1} 

Yurij M.\ Berezansky's contributions to analysis in general, and areas such as functional analysis, operator theory, spectral and inverse spectral theory, harmonic analysis, analysis in spaces of functions of an infinite number of variables, stochastic calculus, mathematical physics, quantum field theory, integration of nonlinear evolution equations, in particular, are legendary and of a lasting nature. The list of fields his ground breaking work changed in dramatic fashion can easily be continued in many directions as is demonstrated by the extraordinary breadth revealed in his highly influential monographs \cite{Be68}--\cite{BSU96a}. Since operator theoretic methods frequently play a role in his research interests, we hope our modest contribution to operator bounds in trace ideals will create some joy for him.

This paper has its origins in the following question: Given self-adjoint operators $S_j$, $j=1,2$, in 
$\cH$, and functions $f$ in an appropriate class, for instance, $f \in C_0^{\infty}(\bbR)$, 
how to control the norm $\|f(S_2) - f(S_1)\|_{\cB(\cH)}$ in terms of the norm of the difference of resolvents, $\big\|(S_2 - z_0 I_{\cH})^{-1} - S_1 - z_0 I_{\cH})^{-1}\big\|_{\cB(\cH)}$, for some 
$z_0 \in \bbC\backslash\bbR$? In particular, the question is just as natural with $\cB(\cH)$ replaced by a trace ideal, $\cB_p(\cH)$, $p \in [1,\infty)$. 

In fact, our interest in these questions stems from computations of the Witten index 
(a suitable extension of the Fredholm index) for a class 
of non-Fredholm model operators 
\begin{equation} 
\bsD_\bsA^{} = \f{d}{dt} + \bsA, \quad 
\dom(\bsD_\bsA^{})= W^{1,2}(\bbR; \cH) \cap \dom(\bsA),    \lb{1.1} 
\end{equation} 
in the Hilbert space  $L^2(\bbR; \cH)$, where 
\begin{align}
&(\bsA f)(t) = A(t) f(t) \, \text{ for a.e.\ $t\in\bbR$,}    \no \\
& f \in \dom(\bsA) = \bigg\{g \in L^2(\bbR;\cH) \,\bigg|\,
g(t)\in \dom(A(t)) \text{ for a.e.\ } t\in\bbR;    \lb{1.2}  \\
& \quad t \mapsto A(t)g(t) \text{ is (weakly) measurable;} \, 
\int_{\bbR} dt \, \|A(t) g(t)\|_{\cH}^2 <  \infty\bigg\},      \no
\end{align}
with $A(t)$, $t\in \bbR$, a family of 
self-adjoint operators  in $\cH$ with asymptotes $A_{\pm}$ (in norm resolvent sense). 
Interesting concrete examles for $A_{\pm}$ are given by massless Dirac-type operators in 
$\cH = L^2(\bbR^d)$, $d \in \bbN$ 
(the latter are known to be non-Fredholm), see, \cite{CGLPSZ14}--\cite{CGPST14}. 
More precisely, given the sequence of self-adjoint operators 
$\bsH_{j,n}$, $\bsH_j$ in $L^2(\bbR; dt; \cH)$, $j=1,2$, $n \in \bbN$, and self-adjoint operators $A_{+,n}$, $A_+$, $A_-$ in $\cH$, $n\in \bbN$, and a Pushnitski-type relation between the 
spectral shift functions $\xi(\, \cdot \,; \bsH_{2,n}, \bsH_{1,n})$ and $\xi(\, \cdot \,; A_{+,n}, A_-)$ 
for the pairs, $(\bsH_{2,n}, \bsH_{1,n})$ and $(A_{+,n}, A_-)$ of the form 
\begin{equation}
\xi(\lambda; \bsH_{2,n}, \bsH_{1,n}) = \begin{cases} 
\f{1}{\pi} \int_{- \lambda^{1/2}}^{\lambda^{1/2}} 
\f{\xi(\nu; A_{+,n}, A_-) \, d \nu}{(\lambda - \nu^2)^{1/2}} \, 
\text{ for a.e.~$\lambda > 0$,} \\
0, \; \lambda < 0, 
\end{cases}  \; n\in\bbN,  \lb{1.3} 
\end{equation} 
we were interested in performing the limit $n \to \infty$ in \eqref{1.3} to obtain the analogous 
relation for the limiting spectral shift functions $\xi(\, \cdot \,; \bsH_2, \bsH_1)$ and 
$\xi(\, \cdot \, ; A_+ , A_-)$ corresponding to the limiting pairs $(\bsH_2, \bsH_1)$ and 
$(A_+, A_-)$, respectively. The latter is instrumental in computing the Witten index for 
$\bsD_\bsA^{}$. The task of performing the limit $n \to \infty$ in \eqref{1.3} is considerably 
complicated since due to the nature of the approximations involved, no suitable bounds on 
$\xi(\, \cdot \,; \bsH_{2,n}, \bsH_{1,n})$ and $\xi(\, \cdot \,; A_{+,n}, A_-)$ (independent of 
$n \in \bbN$) are readily available.  
To circumvent this difficulty one can resort to a distributional approach considering 
\begin{equation}
\int_{\bbR} d \lambda \, \xi(\lambda; \bsH_{2,n}, \bsH_{1,n}) f'(\lambda) 
= \f{1}{\pi} \int_{\bbR} d\nu \, \xi(\nu; A_{+,n},A_-) F'(\nu), \quad n \in \bbN,   \lb{1.4} 
\end{equation}
where $f \in C_0^{\infty}(\bbR)$ is arbitrary, and $F' \in C_0^{\infty}(\bbR)$ is given by 
\begin{equation}
F'(\nu) = \int_{\nu^2}^{\infty} d\lambda \, f'(\lambda) (\lambda - \nu^2)^{-1/2}, \quad \nu \in \bbR. 
\lb{1.5} 
\end{equation}
Focusing now on the left-hand side 
of \eqref{1.4}, one recalls Krein's trace formula,  
\begin{equation}
{\tr}_{\cB_1(L^2(\bbR; \cH))} (f(\bsH_{2,n}) - f(\bsH_{1,n})) = 
 \int_{[0,\infty)} d \lambda \, \xi(\lambda; \bsH_{2,n}, \bsH_{1,n}) f'(\lambda), \quad 
f \in C_0^{\infty}(\bbR).      \lb{1.6}
\end{equation}
Thus, given control of resolvents in the form 
\begin{align}
\begin{split} 
& \lim_{n\to\infty} \big\|\big[(\bsH_{2,n} - z \, \bsI)^{-m_2} - (\bsH_{1,n} - z \, \bsI)^{-m_2}\big]  \\
& \hspace*{1cm} - [(\bsH_2 - z \, \bsI)^{-m_2} - (\bsH_1 - z \, \bsI)^{-m_2}\big]
\big\|_{\cB_1(L^2(\bbR; \cH))} = 0, \quad z \in \bbC \backslash \bbR.    \lb{1.7}
\end{split} 
\end{align}
one can hope to control 
\begin{align}
\begin{split} 
& \lim_{n\to\infty} \|[f(\bsH_{2,n}) - f(\bsH_{1,n})] 
- [f(\bsH_2) - f(\bsH_1)]\|_{\cB_1(L^2(\bbR; \cH))} =0,    \lb{1.8}
\end{split} 
\end{align} 
and hence obtain 
\begin{align}
&  \lim_{n \to \infty} \int_{[0,\infty)} d \lambda \, \xi(\lambda; \bsH_{2,n}, \bsH_{1,n}) f'(\lambda) 
= \lim_{n \to \infty} {\tr}_{\cB_1(L^2(\bbR; \cH))} (f(\bsH_{2,n}) - f(\bsH_{1,n}))  \no \\
& \quad = {\tr}_{\cB_1(L^2(\bbR; \cH))} (f(\bsH_2) - f(\bsH_1)) 
= \int_{[0,\infty)} d \lambda \, \xi(\lambda; \bsH_2, \bsH_1) f'(\lambda), \quad 
f \in C_0^{\infty}(\bbR).       \lb{1.9}
\end{align} 
Together with controlling the limit $n\to\infty$ on the right-hand side of \eqref{1.4}, this leads to  
\begin{equation}
\int_{\bbR} d \lambda \, \xi(\lambda; \bsH_2, \bsH_1) f'(\lambda) 
= \f{1}{\pi} \int_{\bbR} d\nu \, \xi(\nu; A_+,A_-) F'(\nu).    \lb{1.10} 
\end{equation}
Without going into further details we note that \eqref{1.10} in turn can be used to prove the limiting relation in \eqref{1.3} and the latter leads to a compution of the semigroup regularized Witten index,  $W_s(\bsD_\bsA^{})$, of  $\bsD_\bsA^{}$: Assuming that $0$ is a left and a right Lebesgue point of $\xi(\, \cdot \,; A_+,A_-)$, denoting the corresponding values by $\xi(0_{\pm}; A_+,A_-)$, the semigroup regularized Witten index is found to be 
\begin{align}
W_s(\bsD_\bsA^{}) &:= \lim_{t \uparrow \infty} \tr_{L^2(\bbR;\cH)}
\big(e^{- t \bsH_1} - e^{-t \bsH_2}\big)     \no \\
&\,= [\xi(0_-; A_+,A_-) + \xi(0_+; A_+,A_-)]/2, 
\end{align}
see, for instance, \cite{CGLPSZ14}--\cite{CGPST14}. (We here use the semigroup regularized Witten index rather than the resolvent regularised one as the former is applicable in the case 
of $d$-dimensional Dirac-type operators $A_{\pm}$, $d\in\bbN$.) We trust this sufficiently illustrates our interest in using control of resolvents of self-adjoint operators to gain control over their $C_0^{\infty}$-functions. 

We also note a further complication lies in the fact that when studying 
multi-dimensional Dirac-type operators $A_{\pm}$, resolvents alone are not sufficient in the trace class context and hence sufficiently high powers (depending on the space dimension involved) of resolvents have to be employed.   

Our principal tool to gain control over $C_0^{\infty}$-functions of $S$ in terms of (powers of)  
resolvents of $S$ is furnished by a suitable application of almost analytic extensions 
$\widetilde{f}_{\ell,\sigma}$ of $f$ in the 
form of a Davies--Helffer--Sj\"ostrand functional calculus \cite{Da95a}, \cite[Ch.~2]{Da95}, 
\cite[Proposition~7.2]{HS89}, of the form  
\begin{equation}
f(S) = \pi^{-1} \int_{\bbC} dx dy \, 
\f{\partial \wti f_{\ell,\sigma}}{\partial \ol z}(z) (S - z I_{\cH})^{-1},   \lb{1.11} 
\end{equation}
and a refinement due to Khochman \cite{Kh07} of the type 
\begin{equation}
f(S) = \frac{1}{\pi}\int_{\bbC}dxdy\, \frac{\partial \widetilde{f}_{\ell,\sigma}}{\partial \overline{z}}(z) 
(z-z_0)^m (S - z_0I_{\cH})^{-m} (S - zI_{\cH})^{-1}, \quad m \in \bbN \cup \{0\}.     \lb{1.12}
\end{equation}
to be discussed in some detail in Section \ref{s2}. Section \ref{s3} contains our principal results 
and some applications. Finally, Appendix \ref{sA} recalls various useful facts concerning 
(powers of) resolvents. 

We conclude with some comments on the notation employed in this paper: 
Let $\cH$ be a separable complex Hilbert space, $(\cdot,\cdot)_{\cH}$ the scalar product in $\cH$
(linear in the second argument), and $I_{\cH}$ the identity operator in $\cH$.

Next, if $T$ is a linear operator mapping (a subspace of) a Hilbert space into another, then 
$\dom(T)$ and $\ker(T)$ denote the domain and kernel (i.e., null space) of $T$. 
The spectrum and resolvent set 
of a closed linear operator in a Hilbert space will be denoted by $\sigma(\cdot)$ 
and $\rho(\cdot)$, respectively. t

The convergence of bounded operators in the strong operator topology (i.e., pointwise limits) will be denoted by $\slim$. 

The Banach spaces of bounded and compact linear operators on a separable complex Hilbert space $\cH$ are denoted by $\cB(\cH)$ and $\cB_\infty(\cH)$, respectively; the corresponding 
$\ell^p$-based trace ideals will be denoted by $\cB_p (\cH)$, their norms are abbreviated by  
$\|\cdot\|_{\cB_p(\cH)}$, $p \geq 1$. Moreover, $\tr_{\cH}(A)$ denotes the 
corresponding trace of a trace class operator $A\in\cB_1(\cH)$. 

The symbol $C^{\infty}_0(\bbR)$ represents $C^{\infty}$-functions of compact support on 
$\bbR$; continuous functions on $\bbR$ vanishing at infinity are denoted by $C_{\infty}(\bbR)$.

\section{Basic Facts on Almost Analytic Extensions And The Functional Calculus For 
Self-Adjoint Operators} \lb{s2}

In this preparatory section we briefly recall the basics of almost analytic extensions and the 
ensuing functional calculus for self-adjoint operators, following Davies' detailed treatment in \cite{Da95a}, \cite[Ch.~2]{Da95}. 

One introduces the class $S^{\beta}(\bbR)$, $\beta \in \bbR$, consisting of all functions 
$f \in C^{\infty}(\bbR)$ such that 
\begin{equation}
f^{(m)}(x) \underset{|x|\to\infty}{=} \Oh(\big\langle x\rangle^{\beta - m}\big), \quad 
m \in \bbN_0,    \lb{2.1}
\end{equation}
where $\langle z \rangle = \big(|z|^2 + 1\big)^{1/2}$, $z\in \bbC$. Then in obvious notation, with 
``$\cdot$'' denoting pointwise multiplication, 
$S^{\beta}(\bbR) \cdot S^{\gamma}(\bbR) \subseteq S^{\beta + \gamma}(\bbR)$, 
$\beta, \gamma \in \bbR$, and the space 
\begin{equation}
\cA(\bbR) = \bigcup_{\beta < 0} S^{\beta}(\bbR)     \lb{2.2} 
\end{equation} 
is an algebra under pointwise multiplication with 
\begin{equation}
C_0^{\infty}(\bbR) \subset \cA(\bbR).     \lb{2.3} 
\end{equation}
In particular, $f \in \cA(\bbR)$ implies $f \in C_{\infty}(\bbR)$ (the continuous functions 
vanishing at $\pm \infty$) and $f^{(m)} \in L^1(\bbR)$, $m \in \bbN$. 

Given $f \in \cA(\bbR)$, one defines an {\it almost analytic extension} $\wti f_{\ell, \sigma}$, 
of $f$ to $\bbC$ by 
\begin{equation}
\wti f_{\ell,\sigma} (z) = \sigma(x,y) \sum_{k=0}^{\ell} \f{f^{(k)}(x) (iy)^k}{k!}, \quad 
z = x + iy \in \bbC, \; \ell \in \bbN,     \lb{2.4} 
\end{equation} 
where 
\begin{equation}
\sigma (x,y) = \tau(y/\langle x \rangle), \; x, y \in \bbR, \quad \tau \in C_0^{\infty}(\bbR), \; 
\tau (s) = \begin{cases} 1, & |s| \leq 1, \\ 0, & |s| \geq 2. \end{cases}   \lb{2.5}
\end{equation} 
The precise structure of $\wti f_{\ell, \sigma}$ will not be important and other expresssions 
for it are possible (cf., \cite{Da95a}). 

We note the formula 
\begin{align}
\f{\partial \wti f_{\ell,\sigma}}{\partial \ol z}(z) 
&= \f{1}{2} \bigg(\f{\partial \wti f_{\ell,\sigma}}{\partial x}(z) 
+ i \f{\partial \wti f_{\ell,\sigma}}{\partial y}(z)\bigg)   \no \\
&= \f{1}{2} [\sigma_x(x,y) + \sigma_y(x,y)] \sum_{k=0}^{\ell} \f{f^{(k)}(x) (iy)^k}{k!} 
+ \f{1}{2} \sigma(x,y) \f{f^{\ell + 1}(x) (iy)^{\ell}}{\ell !},    \lb{2.6} \\
& \hspace*{9.1cm}  z \in \bbC,    \no 
\end{align}
implying the crucial fact,
\begin{equation}
\bigg|\f{\partial \wti f_{\ell,\sigma}}{\partial \ol z}(x + iy) \bigg| \underset{y \downarrow 0}{=} 
\Oh\big(|y|^{\ell}\big),    \lb{2.7} 
\end{equation} 
in particular, 
\begin{equation}
\f{\partial \wti f_{\ell,\sigma}}{\partial \ol z}(x) = 0, \quad x \in \bbR.    \lb{2.8} 
\end{equation} 

Following Helffer and Sj\"ostrand \cite[Proposition~7.2]{HS89}, particularly, in the form 
presented by Davies \cite{Da95a}, \cite[Ch.~2]{Da95}, one then establishes a 
functional calculus for self-adjoint operators $S$ in a complex, separable Hilbert space 
$\cH$ via the formula
\begin{equation}
f(S) = \pi^{-1} \int_{\bbC} dx dy \, 
\f{\partial \wti f_{\ell,\sigma}}{\partial \ol z}(z) (S - z I_{\cH})^{-1}.   \lb{2.9} 
\end{equation}
Since the integrand is norm continuous, the integral in \eqref{2.9} is norm convergent, 
in particular, one notes that \eqref{2.7} and \eqref{2.8}, together with the standard estimate 
$\big\|(S - z I_{\cH})^{-1}\big\|_{\cB(\cH)} \leq |\Im(z)|^{-1}$, overcome the apparent 
singularity of the integrand in \eqref{2.9} for $z \in \sigma(S) \subseteq \bbR$ (cf.\ also 
\eqref{2.8}).  

The justification for calling this a functional calculus follows upon proving the following 
facts: \\[1mm]
$\bullet$ The left-hand side of \eqref{2.8} is independent of the choice of 
$\ell \in \bbN$ and the precise \hspace*{2.3mm} form of $\sigma$ in \eqref{2.5}. \\[1mm]
$\bullet$ If $f \in C_0^{\infty}(\bbR)$ with $\supp(f) \cap \sigma(S) = \emptyset$, then $f(S)=0$. 
\\[1mm]
$\bullet$ If $f, g \in \cA(\bbR)$, then $(fg)(S) = f(S) g(S)$, $f(S)^* = \ol f(S)$, 
$\|f(S)\|_{\cB(\cH)} = \|f\|_{L^{\infty}(\bbR)}$. \\[1mm]
$\bullet$ Let $z \in \bbC \backslash \bbR$ and $f_z(x) = (x - z)^{-1}$, then $f_z \in \cA(\bbR)$ 
and $f_z(S) = (S - z I_{\cH})^{-1}$. 

In addition, we note that Khochman \cite{Kh07} proved the following extension of \eqref{2.9}:

\begin{lemma} [\cite{Kh07}] \lb{l2.1}
Let $m \in \bbN$, $f\in C_0^{\infty}(\bbR)$, and suppose that $S$ is self-adjoint in $\cH$. Then, 
\begin{equation}
f(S) = \frac{1}{\pi}\int_{\bbC}dxdy\, \frac{\partial \widetilde{f}_{\ell,\sigma}}{\partial \overline{z}}(z) 
(z-z_0)^m (S - z_0I_{\cH})^{-m} (S - zI_{\cH})^{-1}.     \lb{2.10}
\end{equation}
\end{lemma}

We will employ (in fact, rederive) \eqref{2.10} in the proof of Theorem \ref{t3.7}. Next, we discuss another extension focusing on semigroups rather than powers of resolvents.

\begin{lemma} \lb{l2.2} 
Let $t > 0$, $f\in C_0^{\infty}(\bbR)$, and suppose that $S$ is self-adjoint and bounded from below in $\cH$. Then, 
\begin{equation}
f(S) = \frac{1}{\pi} \int_{\bbC}dxdy\, \frac{\partial \widetilde{f}_{\ell,\sigma}}{\partial \overline{z}}(z) e^{tz} e^{- t S}(S - z I_{\cH})^{-1}.      \lb{2.11} 
\end{equation}
\end{lemma}
\begin{proof} 
We start by noting that if $f,g\in C_0^{\infty}(\bbR)$, it is proved in \cite[p.~28]{Da95} that
\begin{equation}\lb{D.17}
\int_{\bbC}dxdy\, \frac{\partial \widetilde{(fg)}_{\ell'',\sigma''}}{\partial \overline{z}}(z) 
(S - z I_{\cH})^{-1} = \int_{\bbC}dxdy\, 
\frac{\partial( \tilde{f}_{\ell',\sigma'}\tilde{g}_{\ell,\sigma})}{\partial \overline{z}}(z) 
(S - z I_{\cH})^{-1}.
\end{equation}

Next, suppose that $f\in C_0^{\infty}(\bbR)$ and let $E_t \in C_0^{\infty}(\bbR)$ denote a function which coincides with $e^{tx}$ on an open interval $I$ with $\supp(f)\subset I$.  Then
\begin{equation}
\widetilde{E}_{t, \ell,\sigma}(z) = \sigma(x,y)e^{tx} \sum_{k=0}^{\ell}\frac{(i t y)^k}{k!},
\quad z=x+iy,\; x\in I.
\end{equation}
Let $f_{\ell',\sigma'}$ denote an almost analytic extension of $f$.  Setting $g= f E_t$, with 
$\widetilde{g}_{\ell'',\sigma''}$ an almost analytic extension of $g$, in light of the identity
$f(S) = g(S) e^{- t S}$, one infers from the Davies--Helffer--Sj\"ostrand functional calculus \eqref{2.9} applied to $g$,
\begin{align}
f(S)&= \frac{1}{\pi}\int_{\bbC} dxdy\, \frac{\partial \widetilde{g}_{\ell'',\sigma''}}{\partial \overline{z}}(z) e^{- t S}(S - z I_{\cH})^{-1}\no\\
&=\frac{1}{\pi}\int_{\bbC}dxdy\, \frac{\partial (\widetilde{f}_{\ell',\sigma'}\widetilde{E}_{t,\ell,\sigma})}{\partial \overline{z}}(z) e^{- t S}(S - z I_{\cH})^{-1}       \no\\
&=\frac{1}{\pi}\int_{\bbC}dxdy\, \frac{\partial \widetilde{f}_{\ell',\sigma'}}{\partial \overline{z}}(z)\widetilde{E}_{t, \ell,\sigma}(z)e^{- t S}(S - z I_{\cH})^{-1}      \no\\
&\quad +\frac{1}{\pi}\int_{\bbC}dxdy\, \widetilde{f}_{\ell',\sigma'}(z)
\frac{\partial \widetilde{E}_{t,\ell,\sigma}}{\partial \overline{z}}(z) e^{- t S}(S - z I_{\cH})^{-1}      \no\\
&=\frac{1}{\pi}\int_{\bbC}dxdy \, \frac{\partial \widetilde{f}_{\ell'',\sigma''}}{\partial \overline{z}}(z)\sigma(x,y)e^{t x} \bigg(\sum_{k=0}^{\ell}\frac{(i t y)^k}{k!}\bigg)e^{- t S}(S - z I_{\cH})^{-1}\no\\
&\quad +\frac{1}{\pi}\int_{\bbC}dxdy\, \widetilde{f}_{\ell',\sigma'}(z)\Bigg\{\frac{1}{2}\bigg[ \sigma_x(x,y)e^{t x} \bigg(\sum_{k=0}^{\ell}\frac{(i t y)^k}{k!}\bigg)    \no \\
& \hspace*{4.05cm} +\sigma(x,y) e^{t x} \bigg(\sum_{k=0}^{\ell}\frac{(i t y)^k}{k!}\bigg)\bigg]\no\\
&\quad +\frac{i}{2}\bigg[ \sigma(x,y)e^{t x} it\bigg(\sum_{k=0}^{\ell-1}\frac{(i t y)^k}{k!}\bigg)+\sigma_y(x,y)e^{t x} \bigg(\sum_{k=0}^{\ell}\frac{(i t y)^k}{k!} \bigg)
\bigg]\Bigg\}    \no \\
& \hspace*{6.6cm} \times e^{- t S}(S - z I_{\cH})^{-1}.    \lb{2.14} 
\end{align}
Exploiting the fact that $f(S)$ is independent of $\ell$, we now take the limit $\ell\to \infty$ in 
\eqref{2.14}.  Since $\widetilde{f}_{\ell',\sigma'}$ has compact support and takes care of the singularity of the resolvent, $e^{- t S}$ is bounded and $z$-independent, and the exponential series converges uniformly on compact sets, one may pass the limit under the integral to obtain
\begin{align}
f(S)&=\frac{1}{\pi}\int_{\bbC}dxdy\, \frac{\partial \widetilde{f}_{\ell',\sigma'}}{\partial \overline{z}}\sigma(x,y)e^{tz} e^{- t S}(S - z I_{\cH})^{-1}\no\\
&\quad +\frac{1}{\pi}\int_{\bbC}dxdy\, \widetilde{f}_{\ell',\sigma'}(z)\frac{\partial(\sigma(x,y)e^{t z})}{\partial \overline{z}}e^{- t S}(S - z I_{\cH})^{-1}\no\\
&=\frac{1}{\pi}\int_{\bbC}dxdy\, \frac{\partial}{\partial\overline{z}}\big(\widetilde{f}_{\ell',\sigma'}(z)\sigma(x,y)e^{t z} \big)e^{- t S}(S - z I_{\cH})^{-1}\no\\
&=\frac{1}{\pi}\int_{\bbC}dxdy\, \frac{\partial}{\partial\overline{z}}\big(\widetilde{f}_{\ell',\widehat{\sigma}}(z)e^{t z} \big)e^{- t S}(S - z I_{\cH})^{-1},
\end{align}
where $\widehat{\sigma}=\sigma'\sigma$ (which corresponds to choosing $\widehat{\tau}=\tau' \tau$).  It is a simple matter to verify that
\begin{equation}
\frac{\partial}{\partial\overline{z}}\big(\widetilde{f}_{\ell',\widehat{\sigma}}(z)e^{t z} \big) = \frac{\partial \widetilde{f}_{\ell,\widehat{\sigma}}}{\partial \overline{z}}e^{t z},
\end{equation}
which then shows
\begin{equation}
f(S) = \frac{1}{\pi} \int_{\bbC}dxdy\, \frac{\partial \widetilde{f}_{\ell',\widehat{\sigma}}}{\partial \overline{z}}e^{t z} e^{- t S}(S - z I_{\cH})^{-1} 
\end{equation}
and hence \eqref{2.11} (renaming $\ell'$ and $\widehat \sigma$).
\end{proof}

Historically, the idea of almost analytic (resp., pseudo-analytic) extensions appeared 
in H\"ormander \cite{Ho69} and Dynkin \cite{Dy75}, \cite{Dy93}, Melin and Sj\"ostrand 
\cite{MS75} (see also \cite[Ch.\ 8]{DS99}, \cite[Sect.\ III.6]{MP89} for expositions and 
\cite{JN94} for an alternative approach). The 
functional calculus was used by Helffer and Sj\"ostrand in their seminal 1989 paper on 
Schr\"odinger operators with magnetic fields \cite{HS89}, which in turn was the basis for 
the systematic treatment by Davies \cite{Da95a}, \cite[Ch.\ 2]{Da95}. Since these early 
developments, there has been a large body of literature in connection with spectral theory 
for Schr\"odinger and Dirac-type operators applying this functional calculus. While a complete 
list of references in this context is clearly beyond the scope of this paper, we want to illustrate 
the great variety of applications that rely on this functional calculus a bit and hence refer to  
\cite{CCO02}, \cite{Di06}, \cite{DD14}, \cite{DP03}, \cite{DP10}, \cite{DS96}, \cite{DZ12}, \cite{FHS10}, \cite{FGM08}, \cite{Ge92}, \cite{Ge08}, \cite{Gr04}, \cite{HS90}, \cite{Kh07}, \cite{ORS13}, \cite{PRS13}, \cite{SZ91}, \cite{Sk92}, and the references cited therein. 

While we here exclusively focus on linear operators in a Hilbert space, this functional calculus applies to operators in Banach spaces with real spectrum, see, for instance, \cite{BF03}, 
\cite{Cl12}, \cite{Da95a}, \cite{Da95b}, \cite{GMP02}, \cite{GP97}. Extensions to the case where 
the spectrum is contained in the unit circle or contained in finitely-many smooth arcs were also treated in \cite{Dy75}.

\section{Some Applications} \lb{s3}

In this section we apply the almost analytic extension method and its ensuing functional 
calculus for self-adjoint operators to derive various norm bounds and convergence properties of 
operators in trace ideals. 

We start with the following estimates established in the proof of \cite[Theorem~2.6.2]{Da95} 
(more precisely, \eqref{3.3} is proved in \cite{Da95}, but then the rest of Lemma \ref{l3.1} 
is obvious):  

\begin{lemma} \lb{l3.1}
Let $z_0 \in \bbC \backslash \bbR$, $f \in \cA(\bbR)$ and suppose that $S_j$, $j=1,2$, 
are self-adjoint in $\cH$. Then 
\begin{align}
\begin{split} 
\| f(S_2) - f(S_1)\|_{\cB(\cH)} & \leq \f{8}{\pi} \int_{\bbC} dx dy \, 
\bigg|\f{\partial \wti f_{2,\sigma}}{\partial \ol z}(z)\bigg| \f{\big[|z_0|^2 + |z|^2\big]}{|\Im(z)|^2}  \\
& \quad \times \big\|(S_2 - z_0 I_{\cH})^{-1} - (S_1 - z_0 I_{\cH})^{-1}\big\|_{\cB(\cH)}.  \lb{3.3} 
\end{split} 
\end{align}
In addition, if for some $p \in [1,\infty)$, 
$\big[(S_2 - z_0 I_{\cH})^{-1} - (S_1 - z_0 I_{\cH})^{-1}\big] \in \cB_p(\cH)$ for some $($and 
hence for all\,$)$ $z_0 \in \bbC \backslash \bbR$, then 
\begin{equation}
[f(S_2) - f(S_1)] \in \cB_p(\cH)    \lb{3.4}
\end{equation}
and 
\begin{align}
\begin{split} 
\| f(S_2) - f(S_1)\|_{\cB_p(\cH)} & \leq \f{8}{\pi} \int_{\bbC} dx dy \, 
\bigg|\f{\partial \wti f_{2,\sigma}}{\partial \ol z}(z)\bigg| \f{\big[|z_0|^2 + |z|^2\big]}{|\Im(z)|^2}  \\
& \quad \times \big\|(S_2 - z_0 I_{\cH})^{-1} - (S_1 - z_0 I_{\cH})^{-1}\big\|_{\cB_p(\cH)}.  \lb{3.5} 
\end{split} 
\end{align}
If $\big[(S_2 - z_0 I_{\cH})^{-1} - (S_1 - z_0 I_{\cH})^{-1}\big] \in \cB_{\infty}(\cH)$, the inclusion 
\eqref{3.4} extends to $p=\infty$. 
\end{lemma}
\begin{proof}
Combining \eqref{2.9}, \eqref{A.-1}, and \eqref{A.0} one obtains,
\begin{equation}
f(S_2) - f(S_1) = \pi^{-1} \int_{\bbC} dx dy \, 
\f{\partial \wti f_{2,\sigma}}{\partial \ol z}(z) \big[(S_2 - z I_{\cH})^{-1}
- (S_1 - z I_{\cH})^{-1}\big]     \lb{3.6} 
\end{equation}
and hence 
\begin{align}
& \|f(S_2) - f(S_1)\|_{\cB(\cH)}    \no \\
& \quad \leq \pi^{-1} \int_{\bbC} dx dy \, 
\bigg|\f{\partial \wti f_{2,\sigma}}{\partial \ol z}(z)\bigg| \big\|(S_2 - z I_{\cH})^{-1}
- (S_1 - z I_{\cH})^{-1}\big\|_{\cB(\cH)}    \no \\
& \quad \leq \pi^{-1} \int_{\bbC} dx dy \, 
\bigg|\f{\partial \wti f_{2,\sigma}}{\partial \ol z}(z)\bigg| 
\big\|(S_2 - z_0 I_{\cH}) (S_2 - z I_{\cH})^{-1}     \no \\ 
& \qquad \times \big[(S_2 - z_0 I_{\cH})^{-1} - (S_1 - z_0 I_{\cH})^{-1}\big]
(S_1 - z_0 I_{\cH}) (S_1 - z I_{\cH})^{-1} \big\|_{\cB(\cH)}    \no \\
& \quad \leq \pi^{-1} \int_{\bbC} dx dy \, 
\bigg|\f{\partial \wti f_{2,\sigma}}{\partial \ol z}(z)\bigg| 
\big\|(S_2 - z_0 I_{\cH}) (S_2 - z I_{\cH})^{-1}\big\|_{\cB(\cH)}     \no \\ 
& \qquad \times \big\|(S_1 - z_0 I_{\cH}) (S_1 - z I_{\cH})^{-1} \big\|_{\cB(\cH)}  
\big\|\big[(S_2 - z_0 I_{\cH})^{-1} - (S_1 - z_0 I_{\cH})^{-1}\big]\big\|_{\cB(\cH)}     \no \\ 
& \quad \leq \bigg(\f{8}{\pi} \int_{\bbC} dx dy \, 
\bigg|\f{\partial \wti f_{2,\sigma}}{\partial \ol z}(z)\bigg| \f{\big[|z_0|^2 + |z|^2\big]}{|\Im(z)|^2}\bigg)  
\no \\
& \qquad \times \big\|\big[(S_2 - z_0 I_{\cH})^{-1} - (S_1 - z_0 I_{\cH})^{-1}\big]\big\|_{\cB(\cH)}. 
\lb{3.7} 
\end{align} 
Precisely the same chain of estimates applies to $\cB(\cH)$ replaced by $\cB_p(\cH)$, relying on the ideal properties of $\cB_p(\cH)$, $p \in [1,\infty)$. The case $p=\infty$ in 
\eqref{3.4} is a consequence of the norm convergent integral on the right-hand side of 
\eqref{3.6}.  
\end{proof}

Combined with a Stone--Weierstrass approximation argument and the fact that 
$\|f(S)\|_{\cB(\cH)} = \|f\|_{L^\infty(\bbR)}$, $f\in \cA(\bbR)$, Lemma \ref{l3.1} yields the 
following well-known fact, recorded, for instance, in \cite[Theorem~2.62]{Da95}, 
\cite[Theorem~V.III.20(a)]{RS80}:

\begin{lemma} \lb{l3.2}
Let $S_n$, $n \in \bbN$, and $S$ be self-adjoint in $\cH$, and suppose that $S_n$ 
converges to $S$ in norm resolvent sense as $n \to \infty$. Then,
\begin{equation}
\lim_{n \to \infty} \|f(S_n) - f(S)\|_{\cB(\cH)} = 0    \lb{3.8} 
\end{equation} 
for all $f \in C_{\infty}(\bbR)$. 
\end{lemma}

\begin{remark} \lb{r3.3}
We note that the functional calculus based on almost analytic extensions is 
not the only possible approach to address estimates such as \eqref{3.3} and \eqref{3.5}. 
As a powerful alternative we mention the theory of double operator integrals (DOI),  
which can prove stronger inequalities of the following type (cf.\ \cite{BS03}, \cite{CGLPSZ14b}, \cite{Ya05}): Given $m \in \bbN$ odd and  $p\in [1,\infty)$, there exist constants 
$a_1,a_2 \in \mathbb{R} \backslash \{0\}$ and $C=C(f,m, a_1, a_2) \in (0,\infty)$ such that 
\begin{align} 
\begin{split} 
& \big\|f(A)-f(B)\|_{\cB_p(\cH)}\leq C\big(\big\|(A-a_1iI_{\cH})^{-m} - (B-a_1iI_{\cH})^{-m}\big\|_{\cB_p(\cH)}         \lb{3.9} \\
&\quad +\big\|(A-a_2iI_{\cH})^{-m} - (B-a_2iI_{\cH})^{-m}\big\|_{\cB_p(\cH)}\big), 
\quad f \in C^{\infty}_0(\bbR),   
\end{split} 
\end{align}
which permits the use of differences of higher powers 
$m \in \mathbb{N}$ of resolvents to control the $\| \cdot \|_{\mathcal{B}_p(\mathcal{H})}$-norm of the left-hand side $[f(A)-f(B)]$ for $f \in C^{\infty}_0(\bbR)$. In fact, this extends to a much larger class of functions $f$, see \cite{CGLNPS16} for details. 

Moreover, repeatedly differentiating 
\begin{equation}
f(\lambda) = \pi^{-1} \int_{\bbC} dx dy \, 
\f{\partial \wti f_{\ell,\sigma}}{\partial \ol z}(z) (\lambda - z)^{-1}, \quad 
\lambda \in \bbR,   \lb{3.10} 
\end{equation}
with respect to $\lambda$ yields
\begin{equation}
f^{(m-1)}(S) = \pi^{-1} (-1)^{m-1} (m-1)! \int_{\bbC} dx dy \, 
\f{\partial \wti f_{\ell,\sigma}}{\partial \ol z}(z) (S - z)^{-m}, \quad 
\lambda \in \bbR.     \lb{3.11} 
\end{equation}
This leads to estimates of the type \eqref{3.9} but with $f$ replaced by $f^{(m-1)}$.  
\hfill $\diamond$
\end{remark}

We recall a useful result:

\begin{lemma} \lb{l3.4}
Let $p\in[1,\infty)$ and assume that $R,R_n,T,T_n\in\cB(\cH)$, 
$n\in\bbN$, satisfy
$\slim_{n\to\infty}R_n = R$  and $\slim_{n\to \infty}T_n = T$ and that
$S,S_n\in\cB_p(\cH)$, $n\in\bbN$, satisfy 
$\lim_{n\to\infty}\|S_n-S\|_{\cB_p(\cH)}=0$.
Then $\lim_{n\to\infty}\|R_n S_n T_n^\ast - R S T^\ast\|_{\cB_p(\cH)}=0$.
\end{lemma}
This follows, for instance, from \cite[Theorem 1]{Gr73}, \cite[p.\ 28--29]{Si05}, or
\cite[Lemma 6.1.3]{Ya92} with a minor additional effort (taking adjoints, etc.).

Next, we describe a typical convergence result: 

\begin{theorem} \lb{t3.5}
Let $S_{j,n}$, $n \in \bbN$, and $S_j$, $j=1,2$, be self-adjoint in $\cH$, and assume 
that $S_{j,n}$ converges in strong resolvent sense as $n \to \infty$ to $S_j$, $j=1,2$,  
respectively. \\
$(i)$ Suppose that for some $($and hence for all\,$)$ 
$z_0 \in \bbC \backslash \bbR$, 
\begin{align}
\begin{split} 
& \lim_{n \to \infty} \big\|\big[(S_{2,n} - z_0 I_{\cH})^{-1} - (S_2 - z_0 I_{\cH})^{-1}\big]   \\
& \hspace*{1cm} 
- \big[(S_{1,n} - z_0 I_{\cH})^{-1} - (S_1 - z_0 I_{\cH})^{-1}\big] \big\|_{\cB(\cH)} 
=0.    \lb{3.12} 
\end{split} 
\end{align}  
Then,
\begin{equation}
\lim_{n \to \infty} \|[f(S_{2,n}) - f(S_{1,n})] - [f(S_2) - f(S_1)]\|_{\cB(\cH)} = 0, \quad 
f \in C_0^{\infty}(\bbR).    \lb{3.13} 
\end{equation}
$(ii)$ Let $p \in [1,\infty)$ and suppose that for some $($and hence for all\,$)$ 
$z_0 \in \bbC \backslash \bbR$, 
\begin{align} 
& \big[(S_{2,n} - z_0 I_{\cH})^{-1} - (S_{1,n} - z_0 I_{\cH})^{-1}\big], 
\big[(S_2 - z_0 I_{\cH})^{-1} - (S_1 - z_0 I_{\cH})^{-1}\big] \in \cB_p(\cH),    \no \\
& \hspace*{10cm} n \in \bbN,    \lb{3.14} 
\end{align}
and 
\begin{align}
\begin{split} 
& \lim_{n \to \infty} \big\|\big[(S_{2,n} - z_0 I_{\cH})^{-1} - (S_2 - z_0 I_{\cH})^{-1}\big]  \\
& \hspace*{1cm} 
- \big[(S_{1,n} - z_0 I_{\cH})^{-1} - (S_1 - z_0 I_{\cH})^{-1}\big] \big\|_{\cB_p(\cH)} 
=0.    \lb{3.15} 
\end{split} 
\end{align}  
Then, 
\begin{equation}
\lim_{n \to \infty} \|[f(S_{2,n}) - f(S_{1,n})] - [f(S_2) - f(S_1)]\|_{\cB_p(\cH)} = 0, \quad 
f \in C_0^{\infty}(\bbR).    \lb{3.16} 
\end{equation}   
\end{theorem}
\begin{proof}
As usual, a combination of identity \eqref{A.-1}, Lemma \ref{l3.4}, and the assumed strong 
resolvent convergence of $S_{j,n}$ to $S_j$ as $n \to \infty$, $j=1,2$, proves sufficiency of 
the conditions \eqref{3.12} and \eqref{3.14} for just one $z_0 \in \bbC \backslash \bbR$. 
Thus, assumption \eqref{3.12} actually implies 
\begin{align}
\begin{split} 
& \lim_{n \to \infty} \big\|\big[(S_{2,n} - z I_{\cH})^{-1} - (S_2 - z I_{\cH})^{-1}\big]   \\
& \hspace*{1cm} 
- \big[(S_{1,n} - z I_{\cH})^{-1} - (S_1 - z I_{\cH})^{-1}\big] \big\|_{\cB(\cH)} 
=0, \quad z \in \bbC\backslash \bbR.    \lb{3.17} 
\end{split} 
\end{align}  

Next, mimicking \eqref{3.6}, one obtains,
\begin{align}
& [f(S_{2,n}) - f(S_{1,n})] - [f(S_2) - f(S_1)]   \no \\ 
& \quad = \pi^{-1} \int_{\bbC} dx dy \, 
\f{\partial \wti f_{2,\sigma}}{\partial \ol z}(z) \Big[ \big[(S_{2,n} - z I_{\cH})^{-1}
- (S_{1,n} - z I_{\cH})^{-1}\big]     \lb{3.18} \\ 
& \hspace*{4.2cm} - \big[(S_2 - z I_{\cH})^{-1} - (S_1 - z I_{\cH})^{-1}\big]\Big], 
\quad f \in C_0^{\infty}(\bbR),     \no 
\end{align} 
and hence,
\begin{align}
& \lim_{n \to \infty} \|[f(S_{2,n}) - f(S_{1,n})] - [f(S_2) - f(S_1)] \|_{\cB(\cH)}  \no \\ 
& \quad \leq \pi^{-1} \lim_{n \to \infty} \int_{\bbC} dx dy \, 
\bigg|\f{\partial \wti f_{2,\sigma}}{\partial \ol z}(z)\bigg|  \Big\|\big[(S_{2,n} - z I_{\cH})^{-1}
- (S_{1,n} - z I_{\cH})^{-1}\big]     \lb{3.19} \\ 
& \hspace*{3cm} - \big[(S_2 - z I_{\cH})^{-1} - (S_1 - z I_{\cH})^{-1}\big]\Big\|_{\cB(\cH)}, 
\quad f \in C_0^{\infty}(\bbR).     \no 
\end{align} 
In this context one observes that $f \in C_0^{\infty}(\bbR)$ implies 
$\wti f_{2,\sigma} \in C_0^{\infty}(\bbC)$. Since the exceptional set $\bbR$ in \eqref{3.17} 
has $dxdy$-measure 
zero (in addition to the fact that by \eqref{2.8}, $\partial \wti f_{2,\sigma}/ \partial \ol z$ vanishes 
on $\bbR$), an application of the Lebesgue dominated convergence theorem to interchange the limit $n \to \infty$ with the integral on the right-hand side of \eqref{3.19} requires establishing 
an $n$-independent integrable majorant of the integrand in \eqref{3.19}. Employing identity 
\eqref{A.-1} and estimate \eqref{A.0}, this majorant can be obtained as follows:
\begin{align}
& \big\|\big[(S_{2,n} - z I_{\cH})^{-1} - (S_{1,n} - z I_{\cH})^{-1}\big] 
- \big[(S_2 - z I_{\cH})^{-1} - (S_1 - z I_{\cH})^{-1}\big]\big\|_{\cB(\cH)}    \no \\
& \quad \leq \big\|(S_{2,n} - z I_{\cH})^{-1} - (S_{1,n} - z I_{\cH})^{-1}\big\|_{\cB(\cH)}    \no \\ 
& \qquad + \big\|(S_2 - z I_{\cH})^{-1} - (S_1 - z I_{\cH})^{-1}\big\|_{\cB(\cH)}    \no \\
& \quad \leq \big\|(S_{2,n} - z_0 I_{\cH})(S_{2,n} - z I_{\cH})^{-1}\big\|_{\cB(\cH)}   \no \\
& \qquad \quad \times \big\|(S_{2,n} - z_0 I_{\cH})^{-1} - (S_{1,n} - z_0 I_{\cH})^{-1}\big\|_{\cB(\cH)} 
\no \\ 
& \qquad \quad \times \big\|(S_{1,n} - z_0 I_{\cH})(S_{1,n} - z I_{\cH})^{-1}\big\|_{\cB(\cH)}      \no \\ 
& \qquad + \big\|(S_2 - z_0 I_{\cH})(S_2 - z I_{\cH})^{-1}\big\|_{\cB(\cH)}   \no \\
& \qquad \quad \times \big\|(S_2 - z_0 I_{\cH})^{-1} - (S_1 - z_0 I_{\cH})^{-1}\big\|_{\cB(\cH)} 
\no \\
& \qquad \quad \times \big\|(S_1 - z_0 I_{\cH})(S_1 - z I_{\cH})^{-1}\big\|_{\cB(\cH)}      \no \\
& \quad \leq 8 \big[|z_0|^2 + |z|^2\big] |\Im(z)|^{-2} 
\Big[\big\|(S_{2,n} - z_0 I_{\cH})^{-1} - (S_{1,n} - z_0 I_{\cH})^{-1}\big\|_{\cB(\cH)}    \no \\ 
& \hspace*{4.5cm} + \big\|(S_2 - z_0 I_{\cH})^{-1} - (S_1 - z_0 I_{\cH})^{-1}\big\|_{\cB(\cH)}\Big] 
\no \\
& \quad \leq 8 \big[|z_0|^2 + |z|^2\big] |\Im(z)|^{-2} C(z_0),     \lb{3.20} 
\end{align}
for some $0 < C(z_0) < \infty$, independent of $n \in \bbN$, since by assumption \eqref{3.12}, 
\begin{equation}
\big\|(S_{2,n} - z_0 I_{\cH})^{-1} - (S_{1,n} - z_0 I_{\cH})^{-1}\big\|_{\cB(\cH)} \leq \wti C(z_0) 
\lb{3.21}
\end{equation}
for some $0 < \wti C(z_0) < \infty$, independent of $n \in \bbN$. Together with the 
properties \eqref{2.7}, \eqref{2.8} of $\partial \wti f_{2,\sigma}/ \partial \ol z$, this establishes 
the sought integrable majorant, independent of $n \in \bbN$, and thus permits the interchange of 
the limit $n \to \infty$ with the integral on the right-hand side of \eqref{3.19}. This completes 
the proof of \eqref{3.13}.

The proof of \eqref{3.16} proceeds exactly along the same lines employing once more the 
ideal properties of $\cB_p(\cH)$, $p \in [1,\infty)$.  
\end{proof}

We remark in passing that double operator integral techniques permit one to enlarge 
the class of functions $f$ to which Theorem \ref{t3.5} applies (cf.\ also Remark \ref{r3.8}).  

\begin{remark} \lb{r3.5a}
The proof to Theorem \ref{t3.5} uses dominated convergence and given \eqref{3.20}, the crucial observation is that
\begin{equation}\lb{D.1}
\int_{\bbC}dxdy\, \bigg|\frac{\partial \widetilde{f}_{\ell,\sigma}}{\partial \overline{z}}(z) 
\bigg|\frac{|z_0|^2+|z|^2}{|\Im(z)|^2}<\infty,
\end{equation}
which is obvious if $f\in C_0^{\infty}(\bbR)$, since then $\widetilde{f}_{\ell,\sigma}$ is compactly supported. However, it is possible to once again prove \eqref{D.1} if $f\in S^{\beta}(\bbR)$ for some $\beta<-1$. Indeed, following \cite[p.~25]{Da95}, and setting
\begin{equation}
U = \{(x,y)\, |\, \langle x \rangle < |y| < 2\langle x \rangle\},\quad V = \{(x,y)\, |\, 0 < |y| < 2\langle x \rangle\},
\end{equation}
one infers
\begin{equation}
|\sigma_x + i \sigma_y|\leq c \langle x \rangle^{-1}\chi_U(x,y),\quad z=x+iy\in \bbC,
\end{equation}
for some constant $c>0$.  Then for $f\in S^{\beta}(\bbR)$,
\begin{align}
\begin{split} 
\bigg|\frac{\partial \widetilde{f}_{\ell,\sigma}}{\partial \overline{z}}(z)\bigg|
&\leq C\Bigg\{\sum_{k=0}^{\ell}\langle x \rangle ^{\beta-k-1}|y|^k \chi_U(x,y)\Bigg\} + C\langle x\rangle^{\beta-\ell-1}|y|^{\ell}\chi_V(x,y),     \\
&\hspace*{6.4cm}\quad z=x+iy\in \bbC,  
\end{split} 
\end{align}
where $C>0$ is an appropriate constant.  Therefore,
\begin{align}
\bigg|\frac{\partial \widetilde{f}_{\ell,\sigma}}{\partial \overline{z}}(z) \bigg|
\frac{|z_0|^2+|z|^2}{|\Im(z)|^2}&\leq C\Bigg\{\sum_{k=0}^{\ell}\langle x \rangle ^{\beta-k-1}|y|^k \chi_U(x,y) \frac{|z_0|^2+|x|^2+|y|^2}{|y|^2}\Bigg\} \no\\
&\quad + C\langle x\rangle^{\beta-\ell-1}|y|^{\ell}\chi_V(x,y) \frac{|z_0|^2+|x|^2+|y|^2}{|y|^2}\no\\
&\leq C2^{\ell}\Bigg\{\sum_{k=0}^{\ell}\langle x \rangle ^{\beta-1} \chi_U(x,y) \frac{|z_0|^2+5\langle x\rangle^2}{\langle x\rangle^2}\Bigg\} \no\\
&\quad + C\langle x\rangle^{\beta-\ell-1}|y|^{\ell-2}\chi_V(x,y) \big\{|z_0|^2+5\langle x\rangle^2\big\}\no\\
&\leq C2^{\ell}\Bigg\{\sum_{k=0}^{\ell}\langle x \rangle ^{\beta-3} \chi_U(x,y) \big\{|z_0|^2+5\langle x\rangle^2\big\}\Bigg\} \no\\
&\quad + C2^{\ell-2}\langle x\rangle^{\beta-3}\chi_V(x,y) \big\{|z_0|^2+5\langle x\rangle^2\big\}\no\\
&\leq \widehat{C}\langle x\rangle^{\beta-1}\big\{\chi_U(x,y) + \chi_V(x,y) \big\},\quad z=x+iy\in \bbC\backslash \bbR,\lb{D.5}
\end{align}
where $\widehat{C} =\widehat{C}(z_0) >0$ is a constant.  Here, we used $|z_0|^2+5\langle x\rangle^2< \widetilde{C}\langle x\rangle^2$, for an appropriate constant $\widetilde{C}=\widetilde{C}(z_0)>0$. Given \eqref{D.5}, \eqref{D.1} holds if $\beta<-1$ since
\begin{align}
&\int_{\bbC}dxdy\, \langle x\rangle^{\beta-1}\big\{\chi_U(x,y) + \chi_V(x,y) \big\} \no\\
&\quad = \int_{-\infty}^{\infty}dx\, \langle x\rangle^{\beta-1}\,  \int_{-\infty}^{\infty}dy\, \big\{\chi_U(x,y) + \chi_V(x,y) \big\}\no\\
&\quad = 2 \int_{-\infty}^{\infty}dx\, \langle x\rangle^{\beta-1}\, \bigg\{ \int_0^{2\langle x\rangle}dy + \int_{\langle x\rangle}^{2\langle x\rangle}dy\bigg\}\no\\
&\quad = 6\int_{-\infty}^{\infty}dx\, \langle x\rangle^{\beta}<\infty.
\end{align}
Thus, the majorant \eqref{D.5} is integrable as long as $f\in S^{\beta}(\bbR)$ for some $\beta<-1$ 
and hence Theorem \ref{t3.5} extends from $C_0^{\infty}(\bbR)$ to $S^{\beta}(\bbR)$, $\beta<-1$.  
\end{remark}

Next, we briefly apply Theorem \ref{t3.5} to a concrete $(1+1)$-dimensional example 
treated in great detail in \cite{CGLPSZ14} and \cite{CGLPSZ14a} by alternative methods. 

\begin{example} \lb{e3.7} 
Assuming the real-valued functions $\phi, \theta$ satisfy  
\begin{align}
& \phi \in AC_{\loc}(\bbR) \cap L^{\infty}(\bbR) \cap L^1(\bbR),  
\; \phi' \in L^{\infty}(\bbR),     \\
\begin{split} 
& \theta \in AC_{\loc}(\bbR) \cap L^{\infty}(\bbR), \; 
\theta' \in L^{\infty}(\bbR) \cap L^1(\bbR),      \\
& \lim_{t \to \infty} \theta (t) = 1, \; \lim_{t \to - \infty} \theta (t) = 0,  
\end{split} 
\end{align} 
we introduce the family of self-adjoint operators $A(t)$, $t \in \bbR$, in $L^2(\bbR)$, 
\begin{equation}
A(t) = - i \f{d}{dx} + \theta(t) \phi, \quad 
\dom(A(t)) = W^{1,2}(\bbR), \; t \in \bbR,     \lb{9.3} 
\end{equation}
and its self-adjoint asymptotes as $t \to \pm \infty$, 
\begin{equation} 
A_+ = - i \f{d}{dx} + \phi, \quad A_- =  - i \f{d}{dx}, \quad 
\dom(A_{\pm}) = W^{1,2}(\bbR).  
\end{equation}
In addition, we introducing the operator $d/dt$ in $L^2\big(\bbR; dt; L^2(\bbR;dx)\big)$  by 
\begin{align}
& \bigg(\f{d}{dt}f\bigg)(t) = f'(t) \, \text{ for a.e.\ $t\in\bbR$,}    \no \\
& \, f \in \dom(d/dt) = \big\{g \in L^2\big(\bbR;dt;L^2(\bbR)\big) \, \big|\,
g \in AC_{\loc}\big(\bbR; L^2(\bbR)\big), \\
& \hspace*{6.3cm} g' \in L^2\big(\bbR;dt;L^2(\bbR)\big)\big\}   \no \\
& \hspace*{2.3cm} = W^{1,2} \big(\bbR; dt; L^2(\bbR; dx)\big).      \label{2.ddtR} 
\end{align} 
Next, we agree to identify $L^2\big(\bbR; dt; L^2(\bbR; dx)\big)$ with $L^2(\bbR^2; dt dx)$ 
$($denoting the latter by $L^2(\bbR^2)$ for brevity\,$)$ and introduce $\bsD_\bsA^{}$ in $L^2(\bbR^2)$ by   
\begin{equation}
\bsD_\bsA^{} = \f{d}{dt} + \bsA,
\quad \dom(\bsD_\bsA^{})= W^{1,2}(\bbR^2),    
\end{equation}
with $\bsA$ defined as in \eqref{1.2} and $A(t)$, $t \in \bbR$, given by \eqref{9.3}. Moreover, 
we introduce the nonnegative, self-adjoint operators $\bsH_j$, $j=1,2$, in $L^2(\bbR^2)$ by
\begin{equation}
\bsH_1 = \bsD_{\bsA}^{*} \bsD_{\bsA}^{}, \quad 
\bsH_2 = \bsD_{\bsA}^{} \bsD_{\bsA}^{*}.
\end{equation} 

As shown in \cite{CGLPSZ14}, the assumptions on $\phi$ and $\theta$ guarantee that \begin{equation}
\big[(A_+ - z I)^{-1} - (A_- - z I)^{-1}\big] \in \cB_1\big(L^2(\bbR)\big), 
\quad z \in \bbC \backslash \bbR     \lb{9.13} 
\end{equation}
$($for simplicity, we adopt the abbreviation $I = I_{L^2(\bbR)}$ throughout this example\,$)$,   
and thus, the spectral shift function $\xi(\, \cdot \, ; A_+, A_-)$ for the pair $(A_+, A_-)$ exists 
and is well-defined up to an arbitrary additive real constant, satisfying 
\begin{equation}
\xi(\, \cdot \, ; A_+, A_-) \in L^1\big(\bbR; (\nu^2 + 1)^{-1} d\nu\big). 
\end{equation}
Introducing $\chi_n(A_-) = n (A_-^2 + n^2 I)^{-1/2}$ and 
$A_{+,n} = A_- + \chi_n(A_-) \phi \chi_n(A_-)$, $n \in \bbN$, the fact 
\begin{equation}
A_{+,n} - A_- = \chi_n(A_-) \phi \chi_n(A_-) \in \cB_1\big(L^2(\bbR)\big), \quad n \in \bbN, 
\end{equation}
implies that also the spectral shift functions 
$\xi(\, \cdot \, ; A_{+,n}, A_-)$, $n \in \bbN$, exist and are uniquely determined by 
\begin{equation}
\xi(\, \cdot \, ; A_{+,n}, A_-) \in L^1(\bbR; d\nu), \quad n \in \bbN. 
\end{equation}
In fact, as has been shown in \cite{CGLPSZ14}, the open constant in $\xi(\, \cdot \, ; A_+, A_-)$ 
can naturally be determined via the limiting procedure 
\begin{equation}
\lim_{n \to \infty} \xi(\nu; A_{+,n}, A_-) = \f{1}{2 \pi} \int_{\bbR} dx \, \phi(x) 
= \xi(\nu; A_+, A_-), \quad \nu \in \bbR.     \lb{9.14} 
\end{equation}
In particular, $\xi(\, \cdot \, ; A_+, A_-)$ turns out to be 
constant in this example. 

Replacing $A(t)$ by $A_n(t) = A_- + \chi_n(A_-) \theta(t) \phi \chi_n(A_-)$, $n \in \bbN$, 
$t \in \bbR$, and hence, $\bsA$ by $\bsA_n$, $\bsD_\bsA^{}$ by $\bsD_{\bsA^{}_n}$, 
$\bsH_j$ by $\bsH_{j,n}$, $j=1,2$, $n \in\bbN$, one verifies the facts,
\begin{align}
& \big[(\bsH_2 - z \, \bsI)^{-1} - (\bsH_1 - z \, \bsI)^{-1}\big] \in 
\cB_1\big(L^2(\bbR^2)\big), \quad z \in \bbC \backslash [0,\infty),      \\
& \big[(\bsH_{2,n}-z \, \bsI)^{-1} - (\bsH_{1,n}-z \, \bsI)^{-1}\big] \in \cB_1\big(L^2(\bbR^2)\big), 
\quad n \in \bbN, \; z \in \bbC \backslash [0,\infty),   
\end{align} 
showing that the spectral shift functions $\xi(\, \cdot \, ; \bsH_2, \bsH_1)$ and 
$\xi(\, \cdot \, ; \bsH_{2,n}, \bsH_{1,n})$ for the pairs $(\bsH_2, \bsH_1)$ and 
$(\bsH_2, \bsH_1)$, $n \in \bbN$, respectively, are well-defined and satisfy 
\begin{equation}
\xi(\, \cdot \, ; \bsH_2, \bsH_1), \, \xi(\, \cdot \, ; \bsH_{2,n}, \bsH_{1,n})
 \in L^1\big(\bbR; (\lambda^2 + 1)^{-1} d\lambda\big), \quad n \in \bbN. 
\end{equation} 
Since $\bsH_j\geq 0$, $\bsH_{j,n} \geq 0$, $n \in \bbN$, $j=1,2$, one uniquely introduces 
$\xi(\,\cdot\,; \bsH_2,\bsH_1)$ and $\xi(\, \cdot \, ; \bsH_{2,n}, \bsH_{1,n})$, $n\in\bbN$, 
by requiring that
\begin{equation}
\xi(\lambda; \bsH_2,\bsH_1) = 0, \quad  
\xi(\, \cdot \, ; \bsH_{2,n}, \bsH_{1,n}) = 0, \quad \lambda < 0, \; n \in \bbN.   
\end{equation}
As shown in \cite{CGLPSZ14}, one can now prove the following intimate connection between 
$\xi(\, \cdot \,; A_{+,n}, A_-)$ and $\xi(\, \cdot \, ; \bsH_{2,n}, \bsH_{1,n})$, $n \in \bbN$, 
the Pushnitski-type formula \cite{GLMST11}, \cite{Pu08}, 
\begin{equation}
\xi(\lambda; \bsH_{2,n}, \bsH_{1,n}) = \f{1}{\pi} \int_{- \lambda^{1/2}}^{\lambda^{1/2}} 
\f{\xi(\nu; A_{+,n}, A_-) d \nu}{(\lambda - \nu^2)^{1/2}} \, \text{ for a.e.~$\lambda > 0$, 
$n\in\bbN$.}    \lb{9.22}
\end{equation}
Moreover, as shown in \cite{CGLPSZ14} and \cite{CGLPSZ14a}, one indeed has the 
convergence property
\begin{align}
\begin{split} 
& \lim_{n\to\infty} \big\|\big[(\bsH_{2,n} - z \, \bsI)^{-1} - (\bsH_{1,n} - z \, \bsI)^{-1}\big]  \\
& \hspace*{1cm} - [(\bsH_2 - z \, \bsI)^{-1} - (\bsH_1 - z \, \bsI)^{-1}\big]
\big\|_{\cB_1(L^2(\bbR^2))} = 0, \quad z \in \bbC \backslash \bbR.    \lb{9.22a}
\end{split} 
\end{align}
Thus, Theorem \ref{t3.5} applies and hence yields,
\begin{equation}
\lim_{n \to \infty} \|[f(\bsH_{2,n}) - f(\bsH_{1,n})] - [f(\bsH_2) - f(\bsH_1)]\|_{\cB_1(L^2(\bbR^2))} 
= 0, \quad f \in C_0^{\infty}(\bbR).    \lb{9.22b} 
\end{equation}   
This, in turn permits one to take the limit $n \to \infty$ in \eqref{9.22}, implying 
\begin{equation}
\xi(\lambda; \bsH_2, \bsH_1) = \f{1}{\pi} \int_{- \lambda^{1/2}}^{\lambda^{1/2}} 
\f{\xi(\nu; A_+, A_-) d \nu}{(\lambda - \nu^2)^{1/2}} \, \text{ for a.e.~$\lambda > 0$, 
$n\in\bbN$.}    \lb{9.22c}
\end{equation}
Equation \eqref{9.22c} combined with 
\eqref{9.14} yields 
\begin{equation}
\xi(\lambda; \bsH_2, \bsH_1) = \xi (\nu; A_+, A_-) = \f{1}{2 \pi} \int_{\bbR} dx \, \phi(x)    \lb{9.23}
\end{equation}
for a.e.~$\lambda > 0$ and a.e.~$\nu \in \bbR$. As a consequence of \eqref{9.23}, the 
semigroup regularized Witten index $W_s(\bsD_\bsA^{})$ of the non-Fredholm operator 
$\bsD_\bsA^{}$ exists and equals 
\begin{equation}
W_s(\bsD_\bsA^{}) = \xi(0_+; \bsH_2, \bsH_1) = 
\xi(0; A_+, A_-) = \f{1}{2 \pi} \int_{\bbR} dx \, \phi(x).     \lb{9.24}
\end{equation}
This yields an alternative proof of the principal Witten index results in \cite{CGLPSZ14} and 
\cite{CGPST14}. 
\end{example}

The following result provides an extension of Theorem \ref{t3.5} to higher powers of 
resolvents (necessitated by applications to $d$-dimensional Dirac-type operators as hinted at 
in the introduction). 
 
\begin{theorem}\lb{t3.7}
Let $S_{j,n}$, $n \in \bbN$, and $S_j$, $j=1,2$, be self-adjoint in $\cH$, and assume that $S_{j,n}$ converges in strong resolvent sense as $n \to \infty$ to $S_j$, $j=1,2$,  respectively.  Suppose that for some $m\in \bbN$ and some 
$p\in [1,\infty)$,
\begin{align}
\big[(S_{2,n}-zI_{\cH})^{-m} - (S_{1,n}-zI_{\cH})^{-m}\big], \big[(S_2-zI_{\cH})^{-m} - (S_1-zI_{\cH})^{-m}\big]\in&\cB_p(\cH),\no\\
z\in \bbC\backslash\bbR,\, n\in \bbN.&       \lb{3.22}
\end{align}
If
\begin{align}
&\lim_{n\rightarrow \infty}\big\|\big[(S_{2,n}-zI_{\cH})^{-m} - (S_{1,n}-zI_{\cH})^{-m}\big] \no\\
&\qquad - \big[(S_2-zI_{\cH})^{-m} - (S_1-zI_{\cH})^{-m}\big]\big\|_{\cB_p(\cH)}=0,\quad z\in \bbC\backslash\bbR, \lb{3.23}
\end{align}
and for some $z_0\in \bbC\backslash \bbR$,
\begin{align}
&(S_{1,n} - z_0I_{\cH})^{-m}\big[(S_{2,n} - z_0I_{\cH})^{-1} - (S_{1,n} - z_0I_{\cH})^{-1} \big],\no\\
&\quad  (S_1 - z_0I_{\cH})^{-m}\big[(S_2 - z_0I_{\cH})^{-1} - (S_1 - z_0)^{-1} \big]\in \cB_p(\cH),\quad n\in \bbN,    \lb{3.23a} 
\end{align}
with
\begin{align}
\lim_{n\rightarrow \infty}&\big\| (S_{1,n} - z_0I_{\cH})^{-m}\big[(S_{2,n} - z_0I_{\cH})^{-1} 
- (S_{1,n} - z_0I_{\cH})^{-1} \big] \no\\
&\quad - (S_1 - z_0I_{\cH})^{-m}\big[(S_2 - z_0I_{\cH})^{-1} - (S_1 - z_0)^{-1} \big]\big\|_{\cB_p(\cH)} = 0,\lb{3.24}
\end{align}
then
\begin{equation}\lb{3.25}
\lim_{n \to \infty} \|[f(S_{2,n}) - f(S_{1,n})] - [f(S_2) - f(S_1)]\|_{\cB_p(\cH)} = 0, \quad 
f \in C_0^{\infty}(\bbR).
\end{equation}
In addition, these results hold upon systematically replacing $\cB_p(\cH)$ by $\cB(\cH)$ in 
\eqref{3.22}--\eqref{3.25}.
\end{theorem}
\begin{proof}
Let $f\in C_0^{\infty}(\bbR)$ be fixed and $\widetilde{f}_{\ell,\sigma}\in C_0^{\infty}(\bbR^2)$ a compactly supported almost analytic extension.  Following a device due to Khochman \cite{Kh07}, 
one introduces 
\begin{equation}\lb{3.26}
g(x) = f(x)(x-z_0)^m,
\end{equation}
concluding $g\in C_0^{\infty}(\bbR)$.  By the Davies--Helffer--Sj\"ostrand functional calculus \eqref{2.9} applied to $g$, one obtains for any self-adjoint operator $S$ in $\cH$, 
\begin{equation}\lb{3.27}
g(S) = \frac{1}{\pi}\int_{\bbC}dxdy\, 
\frac{\partial \widetilde{f}_{\ell,\sigma}}{\partial \overline{z}}(z)(z-z_0)^m (S-zI_{\cH})^{-1},
\end{equation}
and hence, 
\begin{align}
f(S) &= (S - z_0I_{\cH})^{-m}g(S) \no\\
&= \frac{1}{\pi}\int_{\bbC}dxdy\, 
\frac{\partial \widetilde{f}_{\ell,\sigma}}{\partial \overline{z}}(z) (z-z_0)^m 
(S - z_0I_{\cH})^{-m} (S - zI_{\cH})^{-1}.\lb{3.28}
\end{align}
Applying \eqref{3.28} multiple times choosing $H\in \{S_2,S_1,S_{2,n},S_{1,n}\}$, one infers
\begin{align}
&\big\| \big[f(S_{2,n}) - f(S_{1,n}) \big] - \big[f(S_2) - f(S_1) \big]\big\|_{\cB_p(\cH)}\no\\
&\quad \leq \frac{1}{\pi}\int_{\bbC}dxdy\, \bigg|\frac{\partial \widetilde{f}_{\ell,\sigma}}{\partial \overline{z}}(z) (z-z_0)^m\bigg|\no\\
&\qquad \times \big\|\big[(S_{2,n} - z_0I_{\cH})^{-m}(S_{2,n} - zI_{\cH})^{-1} - (S_{1,n} - z_0I_{\cH})^{-m}(S_{1,n} - zI_{\cH})^{-1} \big]\no\\
&\qquad \quad - \big[(S_2 - z_0I_{\cH})^{-m}(S_2 - zI_{\cH})^{-1} - (S_1 - z_0I_{\cH})^{-m}(S_1 - zI_{\cH})^{-1} \big]    \big\|_{\cB_p(\cH)},\no\\
&\hspace*{10cm} n\in \bbN.\lb{3.29}
\end{align}
In order to prove the convergence claim in \eqref{3.25}, the idea is to take the limit $n\to \infty$ and apply dominated convergence in \eqref{3.29}.  However, doing so requires one to obtain an $n$-independent integrable majorant for the expression under the integral in \eqref{3.29} and then to show that the integrand converges to zero pointwise with respect to $z$ as $n\to \infty$.  In order to carry this out, one expresses the difference in the $\|\,\cdot\,\|_{\cB_p(\cH)}$-norm in \eqref{3.29} as follows:
\begin{align}
&\big[(S_{2,n} - z_0I_{\cH})^{-m}(S_{2,n} - zI_{\cH})^{-1} - (S_{1,n} - z_0I_{\cH})^{-m}(S_{1,n} - zI_{\cH})^{-1} \big]\no\\
&\qquad - \big[(S_2 - z_0I_{\cH})^{-m}(S_2 - zI_{\cH})^{-1} - (S_1 - z_0I_{\cH})^{-m}(S_1 - zI_{\cH})^{-1} \big]\no\\
&\quad = (S_{2,n} - z_0I_{\cH})^{-m}(S_{2,n} - zI_{\cH})^{-1} - (S_{1,n} - z_0I_{\cH})^{-m}(S_{2,n} - zI_{\cH})^{-1}\no\\
&\qquad + (S_{1,n} - z_0I_{\cH})^{-m}(S_{2,n} - zI_{\cH})^{-1} - (S_{1,n} - z_0I_{\cH})^{-m}(S_{1,n} - zI_{\cH})^{-1}\no\\
&\qquad - \big[ (S_2 - z_0I_{\cH})^{-m}(S_2 - zI_{\cH})^{-1} - (S_1 - z_0I_{\cH})^{-m}(S_2 - zI_{\cH})^{-1}\no\\
&\qquad \quad + (S_1 - z_0I_{\cH})^{-m}(S_2 - zI_{\cH})^{-1} - (S_1 - z_0I_{\cH})^{-m}(S_1 - zI_{\cH})^{-1}\big]\no\\
&\quad = \big[(S_{2,n} - z_0I_{\cH})^{-m} - (S_{1,n} - z_0I_{\cH})^{-m} \big](S_{2,n} - zI_{\cH})^{-1}\no\\
&\qquad + (S_{1,n} - z_0I_{\cH})^{-m}\big[(S_{2,n} - zI_{\cH})^{-1} - (S_{1,n} - zI_{\cH})^{-1} \big]\no\\
&\qquad -\big\{ \big[ (S_2 - z_0I_{\cH})^{-m} - (S_1 - z_0I_{\cH})^{-m}\big](S_2 - zI_{\cH})^{-1}\no\\
&\qquad \quad + (S_1 - z_0I_{\cH})^{-m}\big[ (S_2 - zI_{\cH})^{-1} - (S_1 - zI_{\cH})^{-1}\big] \big\}\no\\
&\quad =  \big\{\big[(S_{2,n} - z_0I_{\cH})^{-m} - (S_{1,n} - z_0I_{\cH})^{-m} \big](S_{2,n} - zI_{\cH})^{-1}\no\\
&\qquad \quad - \big[ (S_2 - z_0I_{\cH})^{-m} - (S_1 - z_0I_{\cH})^{-m}\big](S_2 - zI_{\cH})^{-1}\big\}\no\\
&\qquad + \big\{(S_{1,n} - z_0I_{\cH})^{-m}\big[(S_{2,n} - zI_{\cH})^{-1} - (S_{1,n} - zI_{\cH})^{-1} 
\lb{3.30} \\
&\qquad \quad - (S_1 - z_0I_{\cH})^{-m}\big[ (S_2 - zI_{\cH})^{-1} - (S_1 - zI_{\cH})^{-1}\big]\big\},\quad z\in \bbC\backslash\bbR,\; n\in \bbN.    \no 
\end{align}
For the first term in braces after the final equality in \eqref{3.30}, one has a bound of the type
\begin{align}
&\big\|\big[(S_{2,n} - z_0I_{\cH})^{-m} - (S_{1,n} - z_0I_{\cH})^{-m} \big](S_{2,n} - zI_{\cH})^{-1}\no\\
&\qquad - \big[ (S_2 - z_0I_{\cH})^{-m} - (S_1 - z_0I_{\cH})^{-m}\big](S_2 - zI_{\cH})^{-1}\big\|_{\cB_p(\cH)}\no\\
&\quad \leq \widetilde{C}(z_0)|\Im(z)|^{-1},\quad z\in \bbC\backslash \bbR,\, n\in \bbN,\lb{3.31}
\end{align}
for a constant $\widetilde{C}(z_0)>0$ which does not depend on $n\in \bbN$ or $z\in \bbC\backslash\bbR$. The estimate in \eqref{3.31} follows at once from the triangle inequality, basic properties of the Schatten--von Neumann trace ideals, the standard resolvent estimate 
\begin{equation}\lb{3.32}
\|(S - zI_{\cH})^{-1}\|_{\cB(\cH)}\leq |\Im(z)|^{-1},\quad  z\in \bbC\backslash \bbR,
\end{equation} 
for an arbitrary self-adjoint operator $S$ in $\cH$, and assumption \eqref{3.23}.  Moreover, by Lemma \ref{3.4}, one also infers
 \begin{align}
&\lim_{n\to \infty}\big\|\big[(S_{2,n} - z_0I_{\cH})^{-m} - (S_{1,n} - z_0I_{\cH})^{-m} \big](S_{2,n} - zI_{\cH})^{-1}\no\\
&\qquad \quad - \big[ (S_2 - z_0I_{\cH})^{-m} - (S_1 - z_0I_{\cH})^{-m}\big](S_2 - zI_{\cH})^{-1}\big\|_{\cB_p(\cH)} =0,\quad z\in \bbC\backslash\bbR.\lb{3.33}
 \end{align}
 For the second term in braces after the final equality in \eqref{3.30}, 
 \begin{align}
 &(S_{1,n} - z_0I_{\cH})^{-m}\big[(S_{2,n} - zI_{\cH})^{-1} - (S_{1,n} - zI_{\cH})^{-1} \no\\
&\qquad \quad - (S_1 - z_0I_{\cH})^{-m}\big[ (S_2 - zI_{\cH})^{-1} - (S_1 - zI_{\cH})^{-1}\big]\no\\
&\quad = (S_1 - z_0I_{\cH})^{-m}\big[(S_1 - zI_{\cH})^{-1} - (S_2 - zI_{\cH})^{-1} \big]\no\\
&\qquad - (S_{1,n} - z_0I_{\cH})^{-m}\big[(S_{1,n} - zI_{\cH})^{-1} - (S_{2,n} - zI_{\cH})^{-1} \big]\no\\
&\quad = \{I_{\cH} + (z-z_0)(S_1 - zI_{\cH})^{-1} \}(S_1 - z_0I_{\cH})^{-m}\big[(S_1 - z_0I_{\cH})^{-1} - (S_2 - z_0I_{\cH})^{-1} \big]\no\\
&\qquad \quad \times \{I_{\cH} + (z-z_0)(S_2 - zI_{\cH})^{-1} \} - \{I_{\cH} + (z-z_0)(S_{1,n} - zI_{\cH})^{-1} \}\no\\
&\qquad \qquad \times(S_{1,n} - z_0I_{\cH})^{-m}\big[(S_{1,n} - z_0I_{\cH})^{-1} - (S_{2,n} - z_0I_{\cH})^{-1} \big]\no\\
&\qquad \qquad \times \{I_{\cH} + (z-z_0)(S_{2,n} - zI_{\cH})^{-1} \},\quad z\in \bbC\backslash\bbR,\, n\in \bbN.\lb{3.34}
\end{align}
Using \eqref{3.32}, one finds for any self-adjoint operator $H$,
\begin{align}
\|I_{\cH} + (z_0 - z)(H - zI_{\cH})^{-1}\|_{\cB(\cH)} &\leq 1 + (|z_0| + |z|)|\Im(z)|^{-1}\no\\
&\leq 2 (|z_0| + |z|)|\Im(z)|^{-1},\quad z\in \bbC\backslash\bbR.\lb{3.35}
\end{align}
As a result of \eqref{3.24}, \eqref{3.34} and \eqref{3.35},
\begin{align}
&\big\|(S_{1,n} - z_0I_{\cH})^{-m}\big[(S_{2,n} - zI_{\cH})^{-1} - (S_{1,n} - zI_{\cH})^{-1} \no\\
&\qquad - (S_1 - z_0I_{\cH})^{-m}\big[ (S_2 - zI_{\cH})^{-1} - (S_1 - zI_{\cH})^{-1}\big]\big\|_{\cB_p(\cH)}\no\\
&\leq \widehat{C}(z_0)\frac{|z_0|^2 + |z|^2}{|\Im(z)|^2},\quad z\in \bbC\backslash\bbR,\, n\in \bbN,\lb{3.36}
\end{align}
for a constant $\widehat{C}(z_0)>0$ which does not depend on $n\in \bbN$ or $z\in \bbC\backslash\bbR$.  Moreover, \eqref{3.24} and another application of Lemma \ref{l3.4} immediately imply
\begin{align}
&\lim_{n\to \infty}\big\|(S_{1,n} - z_0I_{\cH})^{-m}\big[(S_{2,n} - zI_{\cH})^{-1} - (S_{1,n} - zI_{\cH})^{-1} \no\\
&\qquad \quad - (S_1 - z_0I_{\cH})^{-m}\big[ (S_2 - zI_{\cH})^{-1} - (S_1 - zI_{\cH})^{-1}\big]\big\|_{\cB_p(\cH)}=0, \quad z\in \bbC\backslash\bbR.\lb{3.38}
\end{align}
The estimates in \eqref{3.32} and \eqref{3.36} show that away from $\bbR$, which has $dxdy$-measure equal to zero, the integrand in \eqref{3.29} is bounded above by
\begin{align}
\big|z-z_0\big|^m\bigg|\frac{\partial \widetilde{f}_{\ell,\sigma}}{\partial \overline{z}}(z)\bigg| \bigg(\frac{\widetilde{C}(z_0)}{|\Im(z)|} + \widehat{C}(z_0)\frac{[|z_0|^2 + |z|^2]}{|\Im(z)|^2} \bigg),\quad z\in \bbC\backslash\bbR,
\end{align}
which is integrable with $\ell=2$ in light of \eqref{2.7} and the fact that $\widetilde{f}_{\ell,\sigma}$ is compactly supported.  Therefore, taking the limit $n\to \infty$ on both sides of \eqref{3.29} and then applying dominated convergence in combination with \eqref{3.31}, \eqref{3.33}, and \eqref{3.38} yields \eqref{3.25}.

Clearly, the proof remains valid with $\cB_p(\cH)$ replaced by $\cB(\cH)$.
\end{proof}

\begin{remark} \lb{r3.8} 
Although applicable to the Witten index computation described in the introduction, 
Theorem \ref{t3.7} is far from optimal. Indeed, upon communicating Theorem \ref{t3.7} to 
G.\ Levitina, D.\ Potapov, and F.\ Sukochev, they subsequently pointed out to us \cite{LPS15} 
that an application of the 
double operator integral method permits one to extend the classes of functions $f$ to the one 
employed in \cite{Ya05}, and more importantly, the DOI approach permits one to dispense 
with the conditions \eqref{3.23a} and \eqref{3.24} altogether. (Conditions 
\eqref{3.23a} and \eqref{3.24} are clearly an artifact of the resolvent term $(S - zI_{\cH})^{-1}$ 
in formula \eqref{3.28}). This will be further pursued elsewhere \cite{CGLNPS16}.  
\end{remark}

\begin{remark} \lb{r3.6} 
While we exclusively focused on applications to self-adjoint operators $S$, 
as long as $\sigma(T)\subset \bbR$ and the singularity of the resolvent $(T - z I_{\cH})^{-1}$ 
of $T$ as $z$ approaches the 
spectrum is uniformly bounded by $|\Im(z)|^{-N}$ for some fixed $N \in \bbN$, choosing 
$\ell \in \bbN$ sufficiently large in $\wti f_{\ell,\sigma}$, one can handle such classes of 
non-self-adjoint operators $T$, particularly, operators in Banach spaces with real spectrum.  
In fact, a functional calculus for the case of a non-self-adjoint operator $T$ with 
$\sigma(T) \subset \bbR$ and a resolvent that satisfies an estimate of the type
\begin{equation}
\|(T - z I_{\cH})^{-1}\|_{\cB(\cH)} \leq c|\Im(z)|^{-1}\bigg(\frac{\langle z\rangle}{|\Im(z)|} \bigg)^{\alpha},\quad z\in \bbC\backslash\bbR,
\end{equation}
for some $c>0$ and $\alpha\geq 0$, was discussed in \cite{Da95a}, \cite{Da95b}, and in 
subsequent developments in \cite{BF03}, \cite{Cl12}, \cite{GMP02}, \cite{GP97}. Moreover, 
the case where the spectrum is contained in the unit circle or contained in finitely-many smooth arcs was discussed in \cite{Dy75}. 
\hfill $\diamond$
\end{remark}

\appendix
\section{Some Useful Resolvent Identities} \lb{sA}
\renewcommand{\theequation}{A.\arabic{equation}}
\renewcommand{\thetheorem}{A.\arabic{theorem}}
\setcounter{theorem}{0} \setcounter{equation}{0}

In this appendix we recall some well-known, yet useful relations for (powers of) resolvents. 

We start by recalling the well-known identity (see, e.g., \cite[p.~178]{We80}), 
\begin{align}
& (T_2 - z I_{\cH})^{-1} - (T_1 - z I_{\cH})^{-1} = (T_2 - z_0 I_{\cH})(T_2 - z I_{\cH})^{-1}   \no \\
& \quad \times \big[(T_2 - z_0 I_{\cH})^{-1} - (T_1 - z_0 I_{\cH})^{-1}\big] 
(T_1 - z_0 I_{\cH})(T_1 - z I_{\cH})^{-1},     \lb{A.-1} \\
& \hspace*{6.6cm} z, z_0 \in \rho(T_1) \cap \rho(T_2),    \no 
\end{align}
where $T_j$, $j=1,2$, are linear operators in $\cH$ with 
$\rho(T_1) \cap \rho(T_2) \neq \emptyset$. In addition, if $S$ is self-adjoint in $\cH$, we recall
the elementary estimate,  
\begin{align}
\begin{split} 
& \big\|(S - z_0 I_{\cH})(S - z I_{\cH})^{-1}\big\|_{\cB(\cH)} 
= \big\| I_{\cH} + (z - z_0) (S - z I_{\cH})^{-1}\big\|_{\cB(\cH)}     \lb{A.0} \\
& \quad \leq 8^{1/2} \big[|z_0|^2 + |z|^2\big]^{1/2} |\Im(z)|^{-1}, \quad z, z_0 \in \bbC \backslash \bbR.
\end{split} 
\end{align}

In addition, for $m\in \bbN$, we note (cf.\ \cite[p.~315]{Ya92}), 
\begin{align}
&(T_2 - zI_{\cH})^{-(m+1)} - (T_1 - zI_{\cH})^{-(m+1)}\no\\
&\quad = \big[(T_2 - zI_{\cH})^{-m} - (T_1 - zI_{\cH})^{-m} \big](T_1 - zI_{\cH})^{-1}\lb{A.1}\\
&\qquad + (T_2 - zI_{\cH})^{-m}\big[(T_2 - zI_{\cH})^{-1} - (T_1 - zI_{\cH})^{-1} \big],\quad z\in \rho(T_1) \cap \rho(T_2),   \no
\end{align}
and
\begin{align}
&(T_2 - zI_{\cH})^{-(m+1)} - (T_1 - zI_{\cH})^{-(m+1)}\no\\
&\quad  = (T_2 - zI_{\cH})^{-1}\big[(T_2 - zI_{\cH})^{-m} - (T_1 - zI_{\cH})^{-m} \big]\lb{A.2}\\
&\qquad + \big[(T_2 - zI_{\cH})^{-1} - (T_1 - zI_{\cH})^{-1} \big](T_2 - zI_{\cH})^{-m}\no\\
&\qquad - \big[(T_2 - zI_{\cH})^{-1} - (T_1 - zI_{\cH})^{-1} \big]\big[(T_2 - zI_{\cH})^{-m} 
- (T_1 - zI_{\cH})^{-m} \big],\no\\
&\hspace*{8.31cm} z\in \rho(T_1) \cap \rho(T_2).\no
\end{align}

Next, by applying Cauchy's integral formula,
\begin{equation}\lb{A.3}
f^{(k)}(z) = -\frac{k!}{2\pi i} \ointctrclockwise_{\Gamma}d\zeta \, \frac{f(\zeta)}{(\zeta - z)^{k+1}}, 
\quad z \in \Omega,
\end{equation} 
where $f$ is an analytic function in the open set $\Omega \subset \bbC$, and $\Gamma$ is a counterclockwise-oriented contour encompassing the point $z\in \Omega$, to a densely defined, closed linear operator $T$ in $\cH$ with nonempty resolvent set, one obtains a formula for higher powers of the resolvent of $T$ in terms of a fixed lower power as follows, 
\begin{align}
\begin{split} 
(T - zI_{\cH})^{-k} = -\frac{(k-m)!(m-1)!}{2\pi i[(k-1)!]} 
\ointctrclockwise_{\Gamma_z}d\zeta\, (\zeta - z)^{m-k-1}(T - \zeta I_{\cH})^{-m},&\lb{A.4}\\
z\in \rho(H_0),& 
\end{split} 
\end{align}
where for each $z\in \rho(T)$, $\Gamma_z$ is any counterclockwise-oriented circular contour centered at $z$ which does not intersect or encompass points of $\sigma(T)$.

The following lemma (cf.~\cite[p.~210]{Ya92}) states an elementary, yet useful, fact: 

\begin{lemma}\lb{l3.8}
Let $S_j$, $j\in \{1,2\}$ be self-adjoint operators in $\cH$.  If 
\begin{equation}\lb{3.40}
\big[(S_2 - zI_{\cH})^{-m} - (S_1 - zI_{\cH})^{-m}\big]\in \cB_p(\cH),\quad z\in \bbC\backslash\bbR,
\end{equation}
for some $p\in [1,\infty) \cup \{\infty\}$ and some $m\in \bbN$, then
\begin{equation}\lb{3.41}
\big[(S_2 - zI_{\cH})^{-n} - (S_1 - zI_{\cH})^{-n}\big]\in \cB_p(\cH),\quad z\in \bbC\backslash\bbR,\, n\geq m.
\end{equation}
\end{lemma}
\begin{proof}
It suffices to apply the Cauchy-type formula \eqref{A.4} and note that
\begin{align}
&(S_2 - zI_{\cH})^{-n} - (S_1 - zI_{\cH})^{-n} \no\\
&\quad = -\frac{(n-m)!(m-1)!}{2\pi i[(n-1)!]}
\ointctrclockwise_{\Gamma_z}d\zeta\, (\zeta - z)^{m-n-1} \big[(S_2 - \zeta I_{\cH})^{-m} 
- (S_1 - zI_{\cH})^{-m}\big],\no\\
&\hspace*{8.5cm} z\in \bbC\backslash\bbR,\, n\geq m,\lb{3.42}
\end{align}
where $\Gamma_z$ is a counterclockwise-oriented circular contour centered at $z$ that does not intersect $\bbR$.
\end{proof}

\medskip

\noindent 
{\bf Acknowledgments.} We are indebted to Alan Carey, Galina Levitina, Denis Potapov, Fedor Sukochev, Yuri Tomilov, and Dmitriy Zanin for helpful discussions, and particularly to Yuri Tomilov 
for pointing out to us a number of key references in connection with almost analytic extensions. 

\medskip
 
 
\end{document}